\documentclass[11pt,twoside,a4paper]{article}
\usepackage{amsmath,amssymb,amsthm,ifpdf}

\usepackage{graphicx,psfrag}
\newtheorem{theorem}{Theorem}[section]
\newtheorem{lemma}[theorem]{Lemma}
\newtheorem{proposition}[theorem]{Proposition}
\newtheorem{corollary}[theorem]{Corollary}

\theoremstyle{definition}

\theoremstyle{remark}
\newtheorem{remark}[theorem]{Remark}

\newcommand{\R}{\mathbb{R}}

\newcommand{\N}{\mathbb{N}}

\newcommand{\cA}{\mathcal{A}}

\newcommand{\deu}{d_{\operatorname{eu}}}

\newcommand{\al}{\alpha}

\newcommand{\ga}{\gamma}

\newcommand{\om}{\omega}

\newcommand{\si}{\sigma}

\newcommand{\la}{\lambda}

\renewcommand{\phi}{\varphi}

\newcommand{\CAT}{\operatorname{CAT}}

\newcommand{\sm}{\setminus}
\newcommand{\sub}{\subset}

\begin{document}

\title{A Flat Strip Theorem for Ptolemaic Spaces}
\author{Renlong Miao \&  Viktor Schroeder}

\maketitle

%%%%%%%%%%%%%%%%%%%%%%%%%%%%%%%%%%%%%%%%%%%%%%%%%%%%%%%%%%%%%%%%%%%%%%%%%%%%%%%%%%%%%%%%%%%%%%%%%%%%%%%%%%%%%%%%%%%%%%%%%%%%%%%%%%%%%%%%%%%%%%%%%%%%%%%%%%%%%%%%%%

\section{Main Result and Motivation}

A metric space $(X,d)$ is called {\it ptolemaic} or short a PT space, if for
all quadruples of points $x,y,z,w \in X$ the Ptolemy inequality

\begin{equation}\label{eq:PT}
|xy|\,|zw|\ \leq\ |xz|\,|yw|\ +\ |xw|\,|yz|
\end{equation}
holds, where $|xy|$ denotes the distance $d(x,y)$.

We prove a flat strip theorem for geodesic ptolemaic spaces.
Two unit speed geodesic lines $c_0, c_1 :\R\to X$ are called
{\it parallel}, if their distance is sublinear, i.e. if
$\lim_{t\to\infty} \frac{1}{t} d(c_0(t),c_1(t)) = \lim_{t\to -\infty} \frac{1}{t} d(c_0(t),c_1(t)) =0$.

\begin{theorem}\label{thm:main}
Let $X$ be a geodesic PT space which is homeomorphic to
$\R\times [0,1]$, such that the boundary curves are parallel
geodesic lines,
then
$X$ is isometric to a flat strip
$\R\times [0,a] \sub \R^2$ with its euclidean metric.
 \end{theorem}

We became interested in ptolemaic metric spaces because of their relation to the geometry
of the boundary at infinity of $\CAT (-1)$ spaces (compare \cite{FS1}, \cite{buys2}). We therefore think that these spaces
have the right to be investigated carefully.

Our paper is a contribution to the following question

\noindent {\bf Q:} Are proper geodesic ptolemaic spaces $\CAT (0)$-spaces?

We give a short discussion of this question at the end of the paper in section
\ref{sec:4pt}.
Main ingredients of our proof is a theorem of Hitzelberger and Lytchak \cite{HL} about isometric embeddings of
geodesic spaces into Banach spaces and the Theorem of Schoenberg \cite{Sch} characterizing
inner product spaces by the PT inequality.

%%%%%%%%%%%%%%%%%%%%%%%%%%%%%%%%%%%%%%%%%%%%%%%%%%%%%%%%%%%%%%%%%%%%%%%%%%%%%%%%%%%%%%%%%%%%%%%%%%%%%%%%%%

\section{Preliminaries} \label{sec:preliminaries}

In this section we collect the most important basic facts about 
geodesic PT spaces which we will need in our arguments. If we do not provide proofs
in this section,
these can be found in
\cite{FLS}, \cite{FS2}.

Let $X$ be a metric space.
By $|xy|$ we denote the distance between points
$x,y \in X$.
We will always parametrize geodesics proportionally to arclength. Thus
a geodesic in $X$ is a map
$c:I\to X$ with
$|c(t)c(s)|= \la |t-s|$ for all $s,t \in I$ and some constant $\la \ge 0$.
A metric space is called geodesic if every pair of points can be joined by a geodesic.

In addition we will use the following convention in this paper.
If a geodesic is parametrized on $[0,\infty)$ or on
$\R$, the parametrization is {\em always} by arclenth.
A geodesic $c:[0,\infty)\to X$ is called a {\em ray}, a geodesic
$c:\R\to X$ is called a {\em line}.

In the sequel $X$ will always denote a geodesic metric space.
For $x,y \in X$ we denote by
$m(x,y)=\{z\in X \mid\, |xz|=|zy|=\frac{1}{2}|xy| \}$ the set of midpoints of $x$ and $y$.
A subset $C\sub X$ is {\em convex}, if for $x,y \in C$ also
$m(x,y)\sub C$.

A function $f:X \to \R$ is {\em convex} (resp. {\em affine}), if for all
geodesics $c:I\to X$ the map
$f\circ c:I \to \R$ is convex (resp. affine).

The space $X$ is called {\em distance convex} if for all
$p\in X$ the distance function
$d_p=|\cdot \ p|$ to the point $p$ is convex.
It is called {\em strictly distance convex}, if the functions $t \mapsto (d_p\circ c)(t)$ are 
strictly convex whenever $c:I \to X$ is a geodesic with
$|c(t)\,c(s)|>||p\,c(t)|-|p\,c(s)||$ for all $s,t\in I$, i.e., neither $c(t)$ and $c(s)$ being
on a geodesic from $p$ to the other.
This  definition is natural, since
the restriction of $d_p$ to a geodesic segment containing $p$ is never strictly convex.
The Ptolemy property easily implies: 

\begin{lemma} \label{lem-distance-convex}
A geodesic PT space is distance convex.
\end{lemma}

As a consequence, we obtain that for PT metric spaces local geodesics
are geodesics. Here we call a map $c:I\to X$ a {\em local geodesic}, if
for all $t\in I$ there exists a neighborhood $t\in I'\subset I$, such that
$c_{|I'}$ is a geodesic.

\begin{lemma}  \label{lem:localglobalgeodesic} $\cite{FS2}$
If $X$ is distance convex, then every local geodesic
is globally minimizing.
\end{lemma}

In \cite{FLS} we gave examples of PT spaces which are not strictly
distance convex.
However, if the space is proper, then the situation is completely different. \\

\begin{theorem} [\cite{FS2}] \label{thm:sdc} 
A proper, geodesic PT  space is strictly distance convex.
\end{theorem}

Since we have a relatively short proof of this result, we present the proof
in section \ref{sec:str_conv} .

\begin{corollary} [\cite{FLS}] \label{cor:uniquemidpoint} 
Let $X$ be a proper, geodesic PT space.
Then for $x,y \in X$ there exists a unique midpoint
$m(x,y)\in X$.
The midpoint function
$m:X\times X\to X$ is continuous.
\end{corollary}

\begin{corollary} [\cite{FS2}] \label{cor:projection} 
Let $X$ be a proper, geodesic PT space, and $A \subset X$ be a closed and
convex subset.
Then there exists a continuous
projection
$\pi_A:X\to A$.
\end{corollary}

\begin{remark}
For $\CAT(0)$ spaces this projection is always $1$-Lipschitz.
We do not know if $\pi_A$ is
$1$-Lipschitz for general proper geodesic PT spaces.
\end{remark}

The strict convexity of the distance function
together with the properness implies
easily (cf. Corollary \ref{cor:uniquemidpoint})

\begin{corollary}  \label{cor:contdepinitialpoints}
Let $X$ be a proper, geodesic PT
space and let $x,y \in X$.
Then there exists a unique geodesic
$c_{xy}:[0,1]\to X$ from $x$ to $y$ and
the map
$X\times X\times [0,1] \to X$,
$(x,y,t)\mapsto c_{xy}(t)$ is continuous.
\end{corollary}

We call two rays
$c ,c':[0,\infty)\to X$  {\em asymptotic}, if
$\lim_{t\to \infty}\frac{1}{t}|c(t)c'(t)|=0$.

\begin{corollary} \label{cor:uniqueasymptray}
 Let $X$ be a proper geodesic PT space and
$c_1,c_2:[0,\infty) \to X$ asymptotic rays with the same initial point $c_1(0)=c_2(0)=p$.
Then $c_1$=$c_2$.
\end{corollary}

\begin{proof}
 Assume that there exists $t_0 > 0$ such that
$x=c_1(t_0)\neq c_2(t_0)=y$. Let $m=m(x,y)$.
By Theorem \ref{thm:sdc} we have
$|pm|<t_0$. Let $\delta = t_0-|pm|>0$.
For $t>t_0$ consider the points
$x,y,x_t=c_1(t_0+t), y_t=c_2(t_0+t)$.
Note that $\frac{1}{t}|x_ty_t|\to 0$ by assumption.
We write
$|xy_t|=t+\alpha_t$ with $0\le \alpha_t$ and
$|yx_t|=t+\beta_t$ with $0\leq \beta_t$.
The PT inequality applied to the four points gives
$$(t+\alpha_t)(t+\beta_t)\le t^2+ |xy|\,|x_ty_t|$$
and thus $(\alpha_t+\beta_t)\le \frac{1}{t}|x_ty_t|\,|xy| \to 0$.
Thus for $t$ large enough $\alpha_t\le \delta$.
Therefore $|y_tm|\le \frac{1}{2}(|y_tx|+|y_ty|)\le t+\delta/2$,
which gives the contradiction
$(t+t_0)=|py_t|\le|pm|+|my_t|\le (t+t_0-\delta/2)$.
\end{proof}

%%%%%%%%%%%%%%%%%%%%%%%%%%%%%%%%%%%%%%%%%%%%%%%%%%%%%%%%%%%%%%%%%%%%%%%%%%%%

We now collect some results on the Busemann functions of asymptotic 
rays and parallel line. 

$X$ denotes always a geodesic PT space.
Let $c:[0,\infty) \to X$ be a geodesic ray.
As usual we define the {\em Busemann function}
$b_c(x) = \lim_{t\to\infty}(|xc(t)|-t)$.
Since $b_c$ is the limit of the convex functions $d_{c(t)} -t$,
it is convex.

The following proposition implies that, in a PT space, asymptotic rays define (up to
a constant) the same Busemann functions.

\begin{proposition} [\cite{FS2}] \label{prop:sublinearrays} 
Let $X$ be a PT space, let $c_1,c_2:[0,\infty) \to X$ be asymptotic
rays with
Busemann functions $b_i:=b_{c_i}$.
Then 
$(b_1-b_2)$ is constant.
\end{proposition}

Let now $c:\R \to X$ be a geodesic line parameterized by arclength.
Let $c^\pm:[0,\infty)\to X$ be the rays
$c^+(t)=c(t)$ and $c^-(t)=c(-t)$.
Let further $b^\pm:=b_{c^\pm}$.

\begin{lemma} [\cite{FS2}] \label{lem:pos} 
$(b^++b^-)\ge 0$ and
$(b^++b^-)=0$ on the line $c$.

\end{lemma}

We now consider Busemann functions for 
parallel lines.

\begin{proposition} [\cite{FS2}] \label{prop:sublinearlines} 
Let $c_1,c_2:\R\to X$ be parallel lines with
with Busemann functions
$b_1^{\pm}$ and $b_2^{\pm}$.
Then
$(b_1^++b_1^-)=(b_2^++b_2^-)$.
\end{proposition}

\begin{corollary} \label{cor:samebusemannfunction}
 If $c_1,c_2:\R \to X$ are parallel lines. Then there are reparametrizations
of $c_1,c_2$ such that $b_1^+=b_2^+$ and $b_1^-=b_2^-$.
\end{corollary}

\begin{proof}
 Since $b_1^+-b_2^+$ is constant by Proposition \ref{prop:sublinearrays}
we can obviously shift the parametrization of $c_2$ such that
$b_1^+=b_2^+$. It follows now from Proposition \ref{prop:sublinearlines}
that then also $b_1^-=b_2^-$.
\end{proof}

\begin{corollary} \label{cor:parallelaffine}
 Let $X$ be a geodesic space which is foliated by parallels to
to a line $c:\R\to X$; i.e. for any point $x\in X$ there
exists a line $c_x$ parallel to $c$ with $x=c_x(0)$.
Then the Busemann functions $b^{\pm}$ of $c$ are affine.
\end{corollary}

\begin{proof}
We show that $b^++b^-=0$.
Let therefore $x\in X$ and let $b_x^{\pm}$ be the Busemann functions of $c_x$.
By Proposition \ref{prop:sublinearlines} $b^++b^-=b_x^++b_x^-$. Now
$(b_x^++b_x^-)(x)=0$, hence $(b^++b^-)(x)=0$.
Thus the sum of the two convex fuction $b^+$ and $b^-$ ist affine.
It follows that $b^+$ and $b^-$ are affine.
\end{proof}

More generally the following holds:

\begin{corollary} \label{cor:parallelaffine2}
 Let $c:\R\to X$ be a line, then the Busemann functions $b^{\pm}$ are affine
on the convex hull of all points contained on lines parallel to $c$.
\end{corollary}

\begin{proof}
 Indeed the above argument shows that $b^++b^-$ is equal to $0$ on all
parallel lines. Since $b^++b^-$ is convex and $\ge 0$ by Lemma \ref{lem:pos},
$b^++b^-=0$ on the convex hull of all parallel lines. Thus $b^+$ and $b^-$
are affine on this convex hull.
\end{proof}

%%%%%%%%%%%%%%%%%%%%%%%%%%%%%%%%%%%%%%%%%%%%%%%%%%%%%%%%%%%%%%%%%%%%%%%%%%%%%%%%%%%%%%%%%
\section{Proof of the Main Result}

We prove a slightly stronger version of the main Theorem, namely:

\begin{theorem}\label{thm:maingeneralversion}
Let $X$ be a geodesic PT space which is topologically a connected 2-dimensional manifold
with boundary $\partial X$, such that the  
the boundary consists of two parallel geodesic lines. 
Then
$X$ is isometric to a flat strip
$\R\times [0,a] \sub \R^2$ with its euclidean metric.
 \end{theorem}

Using Corollary \ref{cor:samebusemannfunction} we can assume that $\partial X = c(\R) \cup c'(\R)$,
where $c,c':\R\to X$ are parallel lines with the same Busemann functions
$b^{\pm}$. In particular
$b^+(c(t))=b^+(c'(t))=-t$ and $b^-(c(t))=b^-(c'(t))=t$.
Let $a:=|c(0)c'(0)|$ and for
$t\in \R$ let $h_t:[0,a]\to X$ the geodesic from
$c(t)$ to $c'(t)$. We emphazise here, that $h_0$ is parametrized by arclength,
but we do not know, if $h_t$ has unit speed
for $t\ne 0$.
We also define $c_0:=c$ and $c_a:=c'$.
Define $h:\R\times [0,a] \to X$ by
$h(t,s)= h_t(s)$. 
With $H_t$ we denote the set $h_t([0,a])$.
By Corollary \ref{cor:parallelaffine2} the Busemann functions
$b^{\pm}$ are affine on the image of $h$ and thus
$b^+(h(t,s))=-t$ and $b^-(h(t,s))=t$ on $H_t$.
   
We claim that $h$ is a homeomorphism: Clearly $h$ is continuous by Corollary \ref{cor:contdepinitialpoints}.
To show injectivity we note first that $H_t \cap H_{t'} =\emptyset$ for
$t\neq t'$ since $b^+$ has different values on the sets and secondly that for fixed
$t$ the map $h_t$ is clearly injective.
Since $c_0,c_a$ are parallel, i.e. the length of $h_t$ is sublinear, we easily see that
$h$ is a proper map. Since $\partial X$ is in the image of
$h$ and $\R\times (0,a)$, $X\sm \partial X$ are 2-dimensional connected manifolds and $h$ is injective
and proper, we see that $h$ is a homeomorphism.

\begin{lemma} \label{lem:parallelline}
 For all $x\in X$ there exists a unique line
$c_x:\R \to X$ with $c_x$ parallel to $c_0$ and $c_a$
and $c_x(0)=x$.
\end{lemma}

\begin{proof}
 The Uniqueness follows from Corollary \ref{cor:uniqueasymptray}.
To show the existence let $x\in H_{t_0}$. Consider for $i$ large enough the unit speed geodesics $c_i^+:[0,d_i]\to X$
from $x$ to $c_0(i)$, where $d_i=|xc_0(i)|$. By local compactness a subsequence will converge to a limit ray
$c_x^+:[0,\infty) \to X$ with $c_x^+(0)=x$. For topological reasons $c_x^+$ intersects $H_t$ for
$t\ge t_0$. Let $c_x^+(\varphi(t))\in H_t$, then the sublinearity of the length of $H_t$ implies that
$\varphi(t)/t \to 1$ and that $c_x^+$ is asymptotic to $c_0$.
Furthermore the convex function $b^+$ has slope $-1$ on $c_x^+$, i.e.
$b^+(c_x^+(t)) = -t_0 -t$.

In a similar way we obtain a ray $c_x^-:[0,\infty) \to X$ with
$c_x^-(0)=x$, $c_x^-$ asymptotic to $c_0^-$ with
$b^+(c_x^-(t))=-t_0+t$.
Now define $c_x:\R\to X$ by $c_x(t)=c_x^+(t)$ for $t\ge 0$ and
$c_x(t)=c_x^-(-t)$ for $t\leq 0$.
Then $c_x$ is a line since
$$2t \geq |c_x(t)c_x(-t)|\geq |b^+(c_x(t))-b^+(c_x(-t))| = 2t,$$
and hence $|c_x(t)c_x(-t)| =2t$.
\end{proof}

For 
$s\in [0,a]$ let
$c_s:= c_{h(0,s)}$ be the parallel line through $h(0,s)$.
Consider
$c:\R\times [0,a] \to X$, $c(t,s)=c_s(t)$.
This is another parametrization of $X$. Note that $b^+(c_s(t))=-t$.

\begin{remark}
 We do not know in the moment, if
$c(t,s)=h(t,s)$, our final result will imply that.
\end{remark}

\noindent Since we have the foliation of $X$ by the lines $c_s$, we have the property:

(A): If $t,t'\in \R$, $x\in H_t$, then there exist $x'\in H_{t'}$ with $|xx'|=|t-t'|$.

\noindent For $0\leq s\leq a$ we define the {\em fibre distance}
$A_s:X\to \R$ in the following way.
Let $x\in X$, $x=c_{s'}(t')$, i.e. $x\in H_{t'}$.
Then $A_s(x) = \pm |xc_{s}(t')|$, where the sign equals to the sign of
$(s'-s)$. 
Thus $A_s(x)$ is the distance in the fibre
$H_{t'}$ from the point $x$ to the intersection point 
$c_s(\R)\cap H_{t'}$.
Note that by easy triangle inequality arguments $A_s$ is a 2-Lipschitz function.

\input{coordinate.TpX}

We also define for $t\in \R$ the function
$B_t:X\to \R$ by
$B_t(x) =(t'-t)$, when $x\in H_{t'}$.
Note that $B_t$ is 1-Lipschitz and affine, since $b^+$ is
1-Lipschitz and affine.

For fixed $x_0=c_{s_0}(t_0)\in X\sm \partial X$ consider the map
$F_{x_o}:X\to \R^2$ defined by

$$F_{x_0}(x)=(B_{t_0}(x),A_{s_0}(x)).
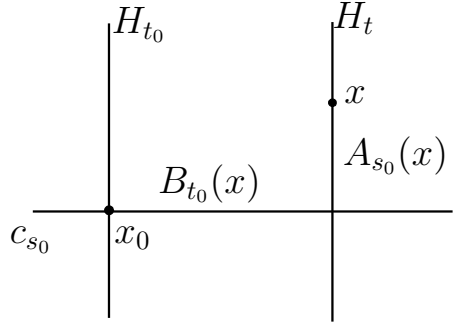
$$

\begin{lemma} \label{lem:bilip}
 $F_{x_0}$ is a bilipschitz map, where $\R^2$ carries the standard euclidean metric
$\deu$, more precisely for all $x,y \in X$ we have
$$\frac{1}{4}|xy|\le \deu(F_{x_0}(x),F_{x_0}(y)) \le 2|xy|. $$
\end{lemma}

\begin{proof}
 Since $B_{t_0}$ is 1-Lipschitz and $A_{s_0}$ is 2-Lipschitz, also
$F_{x_0}$ is 2-Lipschitz.
Now assume
$x\in H_t$, $y\in H_{t'}$.
We claim that
$|F_{x_0}(x)-F_{x_0}(y)|\ge \frac{1}{4}|xy|$. To prove this claim, we can assume that
$|B_{t_0}(x)-B_{t_0}(y)| \leq \frac{1}{4}|xy|$.
By Property (A) there exists $x'\in H_{t'}$ with
$|xx'|=|t-t'|\le \frac{1}{4}|xy|$. Thus $|x'y|\geq \frac{3}{4}|xy|$ and
hence
$$|A_{s_0}(y)-A_{s_0}(x)|\ge |A_{s_o}(y)-A_{s_0}(x')|-|A_{s_0}(x')-A_{s_0}(x)|.$$
Note that 
$$|A_{s_o}(y)-A_{s_0}(x')| =|yx'|$$
 and
$$|A_{s_0}(x')-A_{s_0}(x)|\le 2|x'x|,$$
since $A_{s_0}$ is 2-Lipschitz.
Thus
$$|A_{s_0}(y)-A_{s_0}(x)| \ge |yx'|-2|x'x|\ge \frac{1}{4}|xy|.$$
\end{proof}

 For $\la >0$ we define
$F_{x_0}^{\la}:X\to \R^2$ by
$F_{x_0}^{\la}(x) =\la F_{x_0}(x)$.
Then $F_{x_0}^{\la}:(X,\la d) \to (\R^2,\deu)$ is also a bilischitz with the same constants $\frac{1}{4}$ and $2$
for all
$\la >0$.
Now consider a sequence
$\la_i \to \infty$ and let $d_{\la_i}$ be the metric on
$W_{\la_i}=\la_i\cdot(F_{x_0}(X)) \subset \R^2$ such that
$F_{x_0}^{\la_i}:(X,\la_i d) \to (W_i,d_{\la_i})$ is an isometry. By the above we have
$\frac{1}{2}\deu \le d_{\la_i} \le 4\deu$.

\begin{proposition} \label{prop:diconvergesdeu}
If $\la_i \to \infty$ then
 $d_{\la_i}$ converges uniformly on compact subsets to the standard euclidean distance
$\deu$.
\end{proposition}

\begin{proof}
Since $x_0$ is an inner point of $X$,
$W_{\la_1} \subset W_{\la_2}\subset \cdots$ and $\bigcup W_{\la_i} =\R^2$.
Since $\frac{1}{2}\deu \le d_{\la_i}\le 4\deu$
any subsequence of the integers has itself a subsequence
$i_j\to \infty$ with
$d_{\la_{i_j}}\to d_{\om}$ for some accumulation metric $d_{\om}$ on $\R^2$.
We show that $d_{\om}= \deu$ is always the the euclidean distance and hence
$d_{\la_i}$ will converge to $\deu$.

To prove this we collect some properties of the accumulation metric $d_{\om}$:

(a)  $(\R^2,d_{\om})$ is a geodesic PT space.

(b) By construction $F_{x_0}^{\la}$ maps the geodesic $c_{s_0}$ to the
line $t\mapsto (t,0)$ in $\R^2$ and the geodesic segment $H_t$ to a part of the line
$s\mapsto (\la (t-t_0),s)$.
Therefore $t\mapsto (t,0)$ is a geodesic parametrized by arclength 
in the metric $(\R^2,d_{\om})$ 
and $s\mapsto (t,s)$ is a geodesic parametrized by arclength for all $s$.
Each of these vertical geodesics $s\mapsto (t,s)$ is contained in a level set of
the Busemann function $b_1$ of the line $t\mapsto (t,0)$.
Thus $b_1(t,s)=-t$ and $b_1$
is affine as a limit of affine functions.

(c) The property (A) implies in the limit that for 
$x=(t,s)$ and $t'\in \R$ there exists $y=(t',s')$ with
$|xy|=|t-t'|$.
In particular the lines $s\mapsto (t,s)$ are all parallel.
Thus if $b_2$ is the Busemann function of $s\mapsto (0,s)$,
then this function is affine by Corollary \ref{cor:parallelaffine}.
Note that $b_2(0,s)=-s$ and
$b_2(t,s) = b_2(t,0) -s$.
Since $b_2$ is affine and $t\mapsto (t,0)$ is a geoodesic, we have
$b_2(t,0) = \al t$ for some $\al \in \R$ and hence
$b_2(t,s)=\al t -s$. 

Thus the two affine functions
$b_1$ and $b_2$ separate the points in $(\R^2,d_{\om})$.
It follows by the result of Hitzelsberger-Lytchak \cite{HL}, that
$(\R^2,d_{\om})$ is isometric to a normed vector space. It follows
then from the theorem of Schoenberg \cite{Sch}, that
$(\R^2,d_{\om})$ is isometric to an inner product space.
We claim that the constant $\al$ equals $0$: Since the line
$s\mapsto (0,s)$ lies in the level of the Busemannn function of the
line $t\mapsto (t,0)$ and the space is an inner product space, the two lines
are orthogonal, i.e. $\al =0$.
It now follows easily that $d_{\om} =\deu$.
\end{proof}

Consider now a unit speed geodesic
$\ga:[0,d]\to X$ with 
$\ga(0)=c_{s_0}(t_0) \in X\sm \partial X$.
Since $B_{t_0}$ is affine, we have
$B_{t_0}(\ga(r))=\al r$ for some $\al \in \R$.

\begin{corollary} \label{cor:pythformula}
 With this notation we have
$$\lim_{r\to 0} \frac{A_{s_0}^2(\ga(r))}{r^2} = 1-\al^2.$$
\end{corollary}

\input{limitequality.TpX}

\begin{proof}
Note that $F_{x_0}(\ga(r))=(\al r, A_{s_0}(\ga(r)))$.
By Propossition \ref{prop:diconvergesdeu}
$$\deu(0,F_{x_0}^{1/r}(\ga(r)))\to \frac{1}{r}|x_0\ga(r)|=1.$$
Now 
$$\deu^2(0,F_{x_0}^{1/r}(\ga(r))) = \al^2 + \frac{A_{s_0}^2(\ga(r))}{r^2}.$$
 
\end{proof}

Let now $\si = c_s(\R)$ one of the parallel lines with $0\le s \le a$ considered
as closed convex subset of $X$
We then have the projection
$\pi_{\si}:X\to \si$ from Corollary \ref{cor:projection}.
We show that the projection stays in the same fibre.

\begin{lemma} \label{lem:projinfibre}
$ b^+(\pi_{\si}(x)) = b^+(x)$
\end{lemma}

\begin{proof}
It suffices to prove the result for
$\si =c_s(\R)$, where $0<s<a$, since for $s=0,a$ it then follows by continuity.

Assume that $\pi_{\si}(x)=x_0 \in H_{t_0}$, while
$x\in H_t$. Let
$\ga:[0,d]\to X$ be the unit speed geodesic from
$x_0$ to $x$ where $d=|x_0x|$. Let
$D:X\to [0,\infty)$ be the distance to $\si$, i.e.
$D(x)=|x\pi_{\si}(x)|$. Note that $D(x) \le |A_s(x)|$ and that
 $D(\ga(r))=r$.
Since $b^+$ is affine we have
$b^+(\ga(r))= \al r -t_0$ for some $\al \in \R$.

We have to show that $\al =0$.
By Corollary \ref{cor:pythformula} 
$$\lim_{r\to 0} \frac{A_s^2(\ga(r))}{r^2} = 1-\al^2.$$
If $|\al| \neq 0$ this would imply that for $r > 0$ small enough
$|A_s(\ga(r))| < r =D(\ga(r))$, in contradiction to
$D(x) \le |A_s(x)|$.

\end{proof}

\begin{lemma} \label{lem:equidistant}
For $s_1,s_2 \in [0,1]$ the function
$t\mapsto |c_{s_1}(t)c_{s_2}(t)|$ is constant.
\end{lemma}

\begin{proof}

Let $c=c_{s_1}$ and $c'=c_{s_2}$.

We put
$\mu(t)=|c(t)c'(t)|$.
By Lemma \ref{lem:projinfibre}
$c'(t)$
is a closest to
$c(t)$
point on
$c'(\R)$,
and vice versa,
$c(t)$
is a closest to
$c'(t)$
point on
$c(\R)$
for every
$t\in\R$.
Thus
$|c(t)c'(t')|$, $|c'(t)c(t')|\ge\max\{\mu(t),\mu(t')\}$
for each
$t$, $t'\in\R$.
Applying the Ptolemy inequality to the quadruple
$(c(t),c(t'),c'(t'),c'(t))$,
we obtain 
$$\max\{\mu(t),\mu(t')\}^2\le|c(t)c'(t')||c'(t)c(t')|
  \le\mu(t)\mu(t')+(t-t')^2.$$
We show that
$\mu(a)=\mu(0)$
for every
$a\in\R$.
Assume W.L.G. that
$a>0$
and put
$m=1/\min_{0\le s\le a}\mu(s)$.
Then
$|\mu(t)-\mu(t')|\le m(t-t')^2$
for each
$0\le t,t'\le a$.
Now
$$\mu(a)-\mu(0)=\mu(s)-\mu(0)+\mu(2s)-\mu(s)+\dots+\mu(a)-\mu((k-1)s),$$
where
$s=a/k$
for 
$k\in\N$.
It follows
$|\mu(a)-\mu(0)|\le mks^2=ma^2/k\to 0$
as
$k\to\infty$.
Hence,
$\mu(a)=\mu(0)$.
 
\end{proof}

As a consequence we have
$|c(t,s)c(t,s')|=|s-s'|$ for all $t\in \R$ and of course we also have
$|c(t,s)c(t',s)|=|t-t'|$ for all $s\in \R$.
Note that Lemma \ref{lem:equidistant} also implies the formula 
$A_{s_0}(c(t,s)) = s-s_0$.

We finally want to show that
$|c(t,s)c(t',s')|=\sqrt{|t-t'|^2+|s-s'|^2}$.

\noindent We assume for simplicity $t'\ge t$ and $s'\ge s$.
Let $\ga:[0,d]\to X$ be a unit speed geodesic from
$c(t,s)$ to $c(t',s')$ with $d=|c(t,s)c(t',s')|$.
We can write $\ga(r)=c(\ga_1(r),\ga_2(r))$.
By our assumption $\ga_1$ and $\ga_2$ are nondecreasing.
Since
$\ga_1(r)=B_{t_0}(\ga(r))$ is affine we have
$$\ga_1(r)=t+\frac{t'-t}{d} r.$$
Note that by the above formula for $A_s$ we have for $r_0 ,r_1 \in [0,d]$ 
$$A_{\ga_2(r_0)}(\ga(r_1))= \ga_2(r_1)-\ga_2(r_0).$$
Therefore it follows from Corollary \ref{cor:pythformula} that
for every $r_0 \in [0,d)$ 
$$\lim_{r\to 0} \frac{(\ga_2(r_0+r)-\ga_2(r_0))^2}{r^2} = 1-\frac{(t'-t)^2}{d^2}.$$
This implies that  $\ga_2$ is differentiable with derivative
$$\ga_2'(r_0) = \sqrt{1-\frac{(t'-t)^2}{d^2}},$$
in particular the derivative is constant and therefore
also $\ga_2$ is affine and hence
$$\ga_2(r) = s+ \frac{s'-s}{d}r.$$
The formula for the derivative also implies
$$\frac{s'-s}{d} =\ga_2'(r_0)=\sqrt{1-\frac{(t'-t)^2}{d^2}}$$
which finally shows our claim
$d^2 = (t'-t)^2+(s'-s)^2.$

%%%%%%%%%%%%%%%%%%%%%%%%%%%%%%%%%%%%%%%%%%%%%%%%%%%%%%%%%%%%%%%%%%%%%%%%%%%%%
\section{A short proof of strict convexity} \label{sec:str_conv}

In this section we give a short proof of Theorem \ref{thm:sdc}.
\begin{theorem}
 A proper, geodesic PT metric space is strictly distance convex.
\end{theorem}

For the proof we need the following elementary 
\begin{lemma}
Let $f:[0,a]\to \mathbb{R}$ be a $1$-Lipschitz convex function with $f(0)=0$. For $t>0$ define $g:(0,a]\to \mathbb{R}$
such that $f(t)=tg(t)$. Then $g(0)=\lim_{t\to 0}g(t)$ exists and $-1\leq g(0)\leq 1$.
\end{lemma}

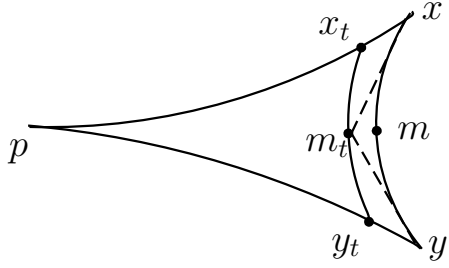
\begin{proof} (of the Theorem)
Since we already know that the distance function $d_p$  is convex, it suffices to show that for
$X,y \in X$ with
$|xy| > ||px|-|py||$
there exits a midpoint $m \in m(x,y)$ such that
for $|pm|<\frac{1}{2}(|px|+|py|)$. 
Using this, it is not hard to see that the midpoint is unique.

We choose a geodesic $px$ from $p$ to $x$ and a geodesic $py$ from
$p$ to $y$.
For $t>0$ small, let $x_t\in px$ and $y_t\in py$
be the points with $|x_tx|=t$ and $|y_ty|=t$.
We choose geodesics $x_ty_t$ from $x_t$ to $y_t$.
For fixed $t$ small enough there exists by continuity a point $w_t\in x_ty_t$ with $|xw_t|=|w_ty|$.
By triangle inequality $|xw_t|=|w_ty|\geq a$.
Using the properness of $(X,d)$, it is
elementary to show that there exists a sequence $t_i \to 0$, such that $\lim_{i\to \infty}w_{t_i}=m$ and $m \in m(x,y)$.
Hence the function $\phi(t_i)=|w_{t_i}m| \to 0$ as $t_i \to 0$.

\input{strictconvex.TpX}

Let us assume to the contrary that
\[
  |pm|=\frac{1}{2} (|px|+|py|)
\]
Let $a=\frac{1}{2}|xy|=|xm|=|my|, b=|px|, c=|pm|, d=|py|$ and we assume w.l.o.g that $b\leq c \leq d$.
By our assumption we have $2c=b+d$.
We write
\[
  |mx_t|=a+ta_x(t), |my_t|+a+ta_y(t)
\]
with function $a_x(t),a_y(t)$ according to the Lemma. The PT inequality applied to $p,x_{t_i},m,y_{t_i}$ gives
\begin{enumerate}
  \item $(a+t_ia_x(t_i))(d-t_i)+(a+t_ia_y(t_i))(b-t_i)\geq c|x_{t_i}y_{t_i}|.$ The sum of the PT inequalities
for $x,x_{t_i},w_{t_i},m$ and $m,w_{t_i},y_{t_i},y$ give that
  \item $a(a+t_ia_x(t_i))+a(a+t_ia_y(t_i))\leq a|x_{t_i}y_{t_i}|+2t_i\phi(t_i)$. From (1) and (2) we obtain
  \item $(a+t_ia_x(t_i))(d-t_i)+(a+t_ia_y(t_i))(b-t_i)\geq c((a+t_ia_x(t_i)+a+t_ia_y(t_i))-2\frac{c}{a}t_i\phi(t_i)$.
Note that by the assumption $2c=b+d$. Thus
  \item $(d-c)a_x(0)+(b-c)a_y(0)\geq 2a$. Since $0\leq (d-c)\leq a$ and $0\geq (b-c)\geq -a$
and $-1\leq a_x(0),a_y(0)\leq 1$ this implies that
\item $a_x(0)=1, a_y(0)=-1$ and $d-c=a, c-b=a$. Hence $|xy|=||px|-|py||$ in contradiction
to the assumption.
\end{enumerate}
\end{proof}

%%%%%%%%%%%%%%%%%%%%%%%%%%%%%%%%%%%%%%%%%%%%%%%%%%%%%%%%%%%%%%%%%%%%%%%%%%%%%%%%%%%%%%%%%%%%%%%

\section{4-Point Curvature Conditions} \label{sec:4pt}

In this section we briefly discuss question ({\bf Q}) stated in the introduction.
We discuss it in the context of conditionts for the distance between four points in a
given metric space. We use the following notation.
Let  $M^4$ be the set of isometry classes of 4-point metric spaces.
For a given metric space $X$ let $M^4(X)$ the set of isometry classes of four point subspaces of $X$.
We consider three inequalities between the distances of four points $x,y,z,w$.

The Ptolemaic inequality

\begin{equation} \label{eq:PT1}
|xy|\,|zw|\ \leq\ |xz|\,|yw|\ +\ |xw|\,|yz|
\end{equation}

The inequality

\begin{equation}\label{eq:QI}
|xy|^2\,+\,|zw|^2\ \leq\ |xz|^2\,+\, |yw|^2\,+\,|xw|^2\,+\,|yz|^2
\end{equation}
which is called the {\it quadrilateral inequality} in \cite{BN} and is equivalent to
the 2-roundness condition of Enflo \cite{E}. 

We also consider
the intermediate inequality
\begin{equation}\label{eq:cosq}
|xy|^2\,+\,|zw|^2\ \leq\ |xz|^2\,+\, |yw|^2\,+\, 2\,|xw|\,|yz|
 \end{equation}
With \cite{BN} we call it the {\it cosq} condition.
Let us denote with
$\cA_{PT},\cA_{QI},\cA_{cosq}$ the isometry classes of spaces in $M^4$, such that for all relabelling of the points
$x,y,z,x$ the conditions 
(\ref{eq:PT1}),(\ref{eq:QI}),(\ref{eq:cosq}) hold respectively.
Since always $2ab\le a^2+b^2$ we clearly have
$\cA_{cosq}\sub\cA_{QI}$, but no other inclusion holds:
The space ${x,y,z,w}$ with $|xy|=2$ and all other distances equal to $1$ shows that
$\cA_{PT} \nsubseteq \cA_{QI}$ and the space
${x,y,z,w}$ with
$|xy|=|zw|=2$, $|xz|=|xw|=1$ and $|yz|=|yw|=a$ with $1<a<2$ and $a$ very close to $2$ shows
$\cA_{cosq}\nsubseteq \cA_{PT}$.

A $\CAT(0)$-space satisfies all condition 
(\ref{eq:PT1}),(\ref{eq:QI}),(\ref{eq:cosq}), i.e. 
$M^4(X) \sub \cA_{cosq} \cap \cA_{PT}$ (see \cite{FLS}, \cite{BFW}).

Berg and Nikolaev (\cite{BN}, compare also\cite{Sa}) proved a beautiful
characterization of $\CAT(0)$ spaces:

A geodesic metric space $X$ is $\CAT(0)$ if and only if al quadruples in
$X$ satisfy the quadrilateral condition (\ref{eq:QI}).

This implies also the following characterization:

A geodesic metric space $X$ is $\CAT(0)$ if and only if al quadruples in
$X$ satisfy the cosq condition (\ref{eq:cosq}).

Formally speaking \cite{BN} proves:
if $X$ is a geodesic metric space with
$M^4(X) \sub \cA_{QI}$, then $X$ is $\CAT(0)$.

The question (Q) asks for a similar characterization
in terms of the PT condition.
In \cite{FLS} we gave examples of geodesic PT which are
not $\CAT(0)$. 
Since these examples are not proper, they leave the question
(Q) open.
Actually in proper geodesic PT the distance function to a point
is strictly convex, see Theorem \ref{thm:sdc}, thus there is some
plausibility for a positive answer to the question.
Our result is another indication in this direction.

Finally we remark that in \cite{FLS} we characterized $\CAT(0)$ sapces
by the property that they are geodesic PT spaces which are in addition Busemann
convex.

%%%%%%%%%%%%%%%%%%%%%%%%%%%%%%%%%%%%%%%%%%%%%%%%%%%%%%%%%%%%%%%%%%%%%%%%%%%%%

\bigskip
\begin{tabbing}

Renlong Miao,\hskip10em\relax \= Viktor Schroeder,\\ 

Institut f\"ur Mathematik,\>
Institut f\"ur Mathematik, \\

Universit\"at Z\"urich,\> Universit\"at Z\"urich,\\
Winterthurer Strasse 190, \>
 Winterthurer Strasse 190, \\

CH-8057 Z\"urich, Switzerland\>  CH-8057 Z\"urich, Switzerland\\

{\tt mrenlong1988@gmail.com}\> {\tt vschroed@math.uzh.ch}\\

\end{tabbing}

\end{document}